\documentclass[twoside,a4paper,11pt]{elsarticle}

\usepackage{amsmath,amsthm,amssymb,enumerate}
\usepackage{epsfig}
\usepackage{amssymb}
\usepackage{amsmath}
\usepackage{amssymb}
\usepackage{amsmath,amsthm}
\usepackage[latin1]{inputenc}
\usepackage[T1]{fontenc}
\usepackage{path}
\usepackage{ae,aecompl}
\usepackage{amsfonts}
\usepackage{amsxtra}
\usepackage{bbm,euscript,mathrsfs}

\usepackage{enumitem}
\setenumerate{label={\rm (\alph{*})}}

\allowdisplaybreaks

\usepackage{amsfonts}
\usepackage{amsxtra}

\usepackage[frak,hyperref,theorem]{paper_diening}

\newcommand{\Div}{\divergence}

\newcommand{\R}{\mathbb R}
\newcommand{\N}{\mathbb N}

\newcommand{\E}{\mathbb E}
\newcommand{\p}{\mathbb P}
\newcommand{\Q}{\mathcal Q}
\newcommand{\F}{\mathcal F}

\newcommand{\HH}{\mathscr H}

\begin{document}

\begin{frontmatter}

\title{Regularity theory for nonlinear systems of SPDEs}

\author{Dominic Breit}
\ead{Breit@math.lmu.de}

\address{LMU Munich,
Mathematical Institute,
Theresienstra\ss e 39,
80333 Munich -- Germany}

  \begin{abstract}
We consider systems of stochastic evolutionary equations of the type
$$d\bfu=\Div\,\bfS(\nabla \bfu)\,dt+\Phi(\bfu)d\bfW_t$$
where $\bfS$ is a non-linear operator, for instance the $p$-Laplacian
$$\bfS(\bfxi)=(1+|\bfxi|)^{p-2}\bfxi,\quad \bfxi\in\R^{d\times D},$$
with $p\in(1,\infty)$ and $\Phi$ grows linearly. We extend known results about 
the deterministic problem to the stochastic situation. First we verify the natural regularity: 
\begin{align*}
\E\bigg[\sup_{t\in(0,T)}\int_{G'}|\nabla\bfu(t)|^2\,dx+\int_0^T\int_{G'}|\nabla\bfF(\nabla\bfu)|^2\,dx\,dt\bigg]<\infty,
\end{align*}
where $\bfF(\bfxi)=(1+|\bfxi|)^{\frac{p-2}{2}}\bfxi$. If we have Uhlenbeck-structure then
$\E\big[\|\nabla\bfu\|_q^q\big]$ is finite
for all $q<\infty$ if the same is true for the initial data.
  \end{abstract}


\end{frontmatter}

\section{Introduction}

We will study existence and regularity of solutions $\bfu:\Q\rightarrow\R^D$,
 $\Q:=(0,T)\times G$ with $T>0$ and $G\subset\R^d$ bounded, to systems of stochastic PDE's of the type
\begin{align}\label{eq:}
\begin{cases}d\bfu&=\Div\,\bfS(\nabla \bfu)\,dt+\Phi(\bfu)d\bfW_t\\
\bfu(0)&=\bfu_0\end{cases}.
\end{align}
Here $\bfS:\R^{d\times D}\rightarrow\R^{d\times D}$ is a non-linear operator and $\bfu_0$ some in general random initial datum. The most famous example is the
 $p$-Laplacian operator
\begin{align}\label{eq:S}
\bfS(\bfxi)=(1+|\bfxi|)^{p-2}\bfxi,\quad \bfxi\in\R^{d\times D},
\end{align}
with $p\in(1,\infty)$.
Equation (\ref{eq:}) is an abbreviation for
\begin{align}\label{eq:strong}
\bfu(t)&=\bfu_0+\int_0^t\Div\,\bfS(\nabla \bfu)\,d\sigma+\int_0^t\Phi(\bfu)\,d\bfW_\sigma
\end{align}
$\p\otimes\mathcal L^1$-a.e.
We assume that $\bfW$ is a Brownian motion with values in a Hilbert space (see section 2 for details) .
We suppose linear growth of $\Phi$ -  roughly speaking $|\Phi(\bfu)|\leq c(1+|\bfu|)$ and $|\Phi'(\bfu)|\leq c$ (for a precise formulation see (\ref{eq:phi}) in section 2).
The motivation for this is an interaction between the solution and the random perturbation caused by
the Brownian motion. For large values of $|\bfu|$ we expect a larger perturbation than for small values.\\
Since the $p$-Laplace equation is a basic problem in non-linear PDE's it can be understood as a
 model-problem to a large class of equations. In view of applications we especially mention the flow of
 Non-Newtonian fluids (see for instance \cite{DRW}, \cite{BrDS}, \cite{TeYo} and \cite{MNRR}) which might be the topic of some future projects. The deterministic equivalent to the equation mentioned above is already well understood, we refer to \cite{Gi}, \cite{Giu} and \cite{Uh} for the stationary case and to \cite{LaUrSo} \cite{DiFr},\cite{Wi} for the evolutionary situation. We also refer to the survey papers \cite{Mi} and \cite{DuMiSt} giving a nice overview.\\
Regarding the stochastic problem there is a lot of literature regarding the existence of solutions
to nonlinear evolutionary equations. The popular variational approach by Pardoux \cite{Pa} for SPDEs provides an existence theory for a quite general class of equations. It requires a Banach space $V$ which is continuously embedded into the Hilbert space $\mathscr H$ on which the equation is considered. The main part of the equation is to be understood in the dual $V^*$. In the situation (\ref{eq:})-(\ref{eq:S}) we have $V=\mathring{W}^{1,p}(G)$ and $\mathscr H=L^2(G)$.
Although this does not include the case $p\leq\frac{2d}{d+2}$ (this bound arises from Sobolev's Theorem) system (\ref{eq:}) can still be treated by slightly modified arguments. For recent developments we refer to  \cite{LiRo} and \cite{PrRo}.\\
However, there is not much literature about the regularity for nonlinear stochastic problems like (\ref{eq:}). Certain regularity results about nonlinear stochastic PDEs are known:
\begin{itemize}
\item In \cite{Ho1} and \cite{Ho2} semilinear stochastic PDEs are considered, were also regularity
statements are shown. Anyway, the elliptic part of the equations studied there is still linear.
\item Zhang \cite{Zh} observes non-linear stochastic PDEs but only in space-dimension one.
\item Very recently the regularity of certain nonlinear parabolic systems with stochastic perturbation was investigated in \cite{BeFl}. The results are $C^\alpha$-estimates
for the solution under a quadratic growth assumption. 
\end{itemize}
The literature dedicated to the regularity theory for linear SPDEs is quite extensive, we refer to \cite{Kr}, \cite{KrRo1}, \cite{KrRo2}, \cite{Fl1} and the references therein.
The situation in the non-linear case is different, as explained above and to our best knowledge there is nothing about regularity  for the stochastic $p$-Laplacian system.
Hence this is the aim of the present paper. We will prove the following statements:
\begin{itemize}
\item The weak solution $\bfu$ is a strong solution to (\ref{eq:}) and it holds
\begin{align}
\E\bigg[\sup_{t\in(0,T)}\int_{G'}|\nabla\bfu(t)|^2\,dx+\int_0^T\int_{G'}|\nabla\bfF(\nabla\bfu)|^2\,dx\,dt\bigg]<\infty\label{neu}
\end{align}
for all $G'\Subset G$, where $\bfF(\bfxi)=(1+|\bfxi|)^{\frac{p-2}{2}}\bfxi$ (see Theorem \ref{thm:2.3} and Theorem and \ref{thm:sub2}).
\item Let $\bfS(\bfxi)=\nu(|\bfxi|)\bfxi$ for $\nu:[0,\infty)\rightarrow[0,\infty)$\footnote{the so-called Uhlenbeck-structure, see \cite{Uh}} and $p>2-\frac{4}{d}$. Then
the strong solution $\bfu$ satisfies
\begin{align}\label{eq:moser}
\E\bigg[\int_0^T\int_{G'}|\nabla\bfu|^q\,dx\,dt\bigg]<\infty
\end{align}
for all $G'\Subset G$ and all $q<\infty$ (see Theorem \ref{thm:4.3}).
\end{itemize}

\begin{remark}
\begin{enumerate}
\item The estimate in (\ref{neu}) is the natural extension of the results for non-linear PDE's in the deterministic situation to the stochastic setting, see \cite{DuMiSt} (chapter 5). In the deterministic case, it is also
quite standard to get regularity results in time: testing with $\partial_t^2\bfu$ gives $\partial_t\bfu\in L^\infty(L^2)$ and $\partial_t\bfF(\nabla\bfu)\in L^2(\Q)$. Due to the appearance of the Brownian motion such a results cannot be true for the stochastic problem.
\item We consider only the non-degenerated case, see (\ref{eq:Sp}). However, the regularity estimates are independent of $\Lambda$ which means it is possible to obtain results for the degenerate case via approximation.
\item It is only a technical matter to assume general Dirichlet boundary conditions. In order to keep the proofs easier, we assume them to be zero.
\item A lot of other statements which are known in the deterministic situation are still open for the stochastic problem. For instance partial regularity for the parabolic problems with $p$-structure which is shown in \cite{DuMi} via the $\mathcal A$-caloric approximation method.
\item The proof of (\ref{eq:moser}) is based on Moser iteration (see for instance \cite{GT}, ch. 8.5, for a nice presentation in the stationary deterministic case). Moser iteration in the stochastic setting also appears
in \cite{DMS1}-\cite{DMS3}. The authors study estimates and maximum princples for the solution to
SPDEs with a linear operator in the main part. This paper is concerned with gradient estimates for nonlinear systems of SPDEs.
\end{enumerate}
\end{remark}

Our procedure is as follows:
In section 3 we study the case $p\geq2$.
We apply the difference quotient method to gain higher differentiability and the corresponding estimates in the superquadratic case (section 4). Since this does not work immediately if $p<2$ we approximate by a quadratic problem and show uniform estimates.
We have to combine the techniques from non-linear PDE's with stochastic calculus for martingales. Note that it is not possible to work directly with test functions. Instead we apply It\^{o}'s formula to certain functions of $\bfu$.\\
Finally we prove arbitrarily high integrability of $\nabla\bfu$ under special structure assumptions. This is done by Moser iteration.

\section{Probability framework}
Let $(\Omega,\mathcal F,\p)$ be a probability space equipped with a filtration $\left\{\mathcal F_t,\,\,0\leq t\leq T\right\}$, which is a nondecreasing
family of sub-$\sigma$-fields of $\mathcal F$, i.e. $\mathcal F_s\subset\mathcal F_t$ for $0\leq s\leq t\leq T$. We further assume that $\left\{\mathcal F_t,\,\,0\leq t\leq T\right\}$ is right-continuous and $\mathcal F_0$ contains all the $\p$-negligible events in $\mathcal F$.\\
For a Banach space $(X,\|\cdot\|_X)$ we denote for $1\leq p<\infty$ by $L^p(\Omega,\mathcal F,\p;X)$ the Banach space of all $\F$-measurable functions $v:\Omega\rightarrow X$ such that
\begin{align*}
\E\big[\|v\|_X^p\big]<\infty,
\end{align*}
where the expectation is taken w.r.t. $(\Omega,\mathcal F,\p)$.\\
Let $U,\HH$ be two separable Hilbert spaces and let $(\bfe_k)_{k\in\N}$ be an orthonormal basis of $U$. We denote by $L_2(U,\HH)$ the set of Hilbert-Schmidt operators from $U$ to $\HH$.  
Throughout the paper we consider a cylindrical Wiener process
$\bfW=(\bfW_t)_{t\in[0,T]}$ which has the form
\begin{align}\label{eq:W}
\bfW(\sigma)=\sum_{k\in\N}\beta_k(\sigma)\bfe_k
\end{align}
with a sequence $(\beta_k)$ of independent real valued Brownian motions on $(\Omega,\mathcal F,\p)$. Now 
\begin{align*}
\int_0^t \psi(\sigma)d\bfW_\sigma,\quad \psi\in L^2(\Omega,\F,\p;L^2(0,T;L_2(U,\HH))),
\end{align*}
with $\psi$ progressively  $(\F_t)$-measurable,
defines a $\p$-almost surely continuous $L^2(\Omega)^D$ valued $\mathcal F_t$-martingale.\footnote{for stochastic calculus in infinite dimensions we refer to \cite{PrZa}} Moreover, we can multiply with test-functions since
 \begin{align*}
\bigg\langle\int_0^t \psi(\sigma)d\bfW_\sigma,\bfphi\bigg\rangle_\HH=\sum_{k=1}^\infty \int_0^t\langle\psi(\sigma)( \bfe_k),\bfphi\rangle_\HH\,d\beta_k(\sigma),\quad \bfphi\in \HH,
\end{align*}
is well-defined (the series converges in $L^2(\Omega,\F,\p; C[0,T])$).\\
Our actual aim is the study of the system (\ref{eq:}), where $\bfH=L^2(G)$, $V=\mathring{W}^{1,p}(G)$:
\begin{align}\label{eq:auxp}
\begin{cases}d\bfu&=\Div\,\bfS(\nabla \bfu)\,dt+\Phi (\bfu)\,d\bfW_t\\
\bfu(0)&=\bfu_0\end{cases},
\end{align}
where $\bfS:\R^{d\times D}\rightarrow \R^{d\times D}$ is $C^1$ and fulfils
\begin{align}\label{eq:Sp}
\lambda (1+|\bfxi|)^{p-2}|\bfzeta|^2\leq D\bfS(\bfxi)(\bfzeta,\bfzeta)\leq \Lambda(1+|\bfxi|)^{p-2} |\bfzeta|^2
\end{align}
for all $\bfxi,\bfzeta\in\R^{d\times D}$ with some positive constants $\lambda,\Lambda$ and $p\in(1,\infty)$. Suppose that $\Phi$ satisfies (\ref{eq:phi}).
\begin{definition}[weak solution]
\label{def:weakp}
$\left.\right.$\\
Let $\bfW$ be a Brownian motion as in (\ref{eq:W}) on a probability space $(\Omega,\F,\p)$ with filtration $(\F_t)$.
A function $\bfu\in L^2(\Omega,\F,\p;L^\infty(0,T;L^2(G)))\cap L^p(\Omega,\F,\p; L^p(0,T;\mathring{W}^{1,p}(G)))$ which is progressively $(\F_t)$-measureable is called a weak solutions to (\ref{eq:}) if
for every $\bfphi\in C^\infty_0(G)$ it holds for a.e. $t$
\begin{align*}
\int_G\bfu(t)\cdot\bfvarphi\,dx &+\int_G\int_0^t\bfS(\nabla \bfu(\sigma)):\nabla\bfphi\,dx\,d\sigma\\&=\int_G\bfu_0\cdot\bfvarphi\,dx+\int_G\int_0^t\Phi(\bfu)\,d\bfW_\sigma\cdot \bfvarphi\,dx
\end{align*}
$\p$-almost surely.
\end{definition}
In order to show regularity of solutions we suppose the following linear growth assumptions on $\Phi$ (following \cite{Ho1}): For each $\bfz\in L^2(G)$ there is a mapping $\Phi(\bfz):U\rightarrow L^{2}(G)^D$ defined by $\Phi(\bfz)\bfe_k=g_k(\cdot,\bfz(\cdot))$. In particular, we suppose
that $g_k\in C^1(G\times\R^D)$ and the following conditions
\begin{align}\label{eq:phi}\begin{aligned}
\sum_{k\in\N}|g_k(x,\bfxi)|^2 \leq c(1+|\bfxi|^2)&,\quad
\sum_{k\in\N}|\nabla_{\bfxi} g_k(x,\bfxi)|^2\leq c,\quad\bfxi\in\R^D,\\
\sum_{k\in\N}|\nabla_x g_k(x,\bfxi)|^2 &\leq c(1+|\bfxi|^2).
\end{aligned}
\end{align}
\begin{definition}[strong solution]
\label{def:strong}
$\left.\right.$\\
A weak solution is called a strong solutions to (\ref{eq:})\\ if $\Div\bfS(\nabla \bfu)\in L^1(\Omega,\F,\p;L^1_{loc}(\Q))$ and
\begin{align*}
\bfu(t)&=\bfu_0+\int_0^t\Div\,\bfS(\nabla \bfu)\,d\sigma+\int_0^t\Phi(\bfu)\,d\bfW_\sigma
\end{align*}
holds $\p\otimes \mathcal L^{d+1}$-a.e.
\end{definition}

\section{Regularity for $p\geq2$}
Throughout this section we study problems of the type (\ref{eq:}) with (\ref{eq:Sp}) for $p\geq2$. In the following section we consider subquadratic problems regularized by quadratic ones.

\begin{theorem}[Regularity]\label{thm:2.3}
$\left.\right.$\\
Assume $\bfu_0\in L^2(\Omega,\mathcal F_0,\p,\mathring{W}^{1,2}(G))$, (\ref{eq:Sp}) with $p\geq2$ and (\ref{eq:phi}). 
Then the unique weak solution $\bfu$ to (\ref{eq:}) is a strong solution and satisfies
\begin{align*}
\E\bigg[\sup_{t\in(0,T)}\int_{G'}|\nabla\bfu(t)|^2\,dx+\int_0^T\int_{G'}|\nabla\bfF(\nabla\bfu)|^2\,dx\,dt\bigg]<\infty
\end{align*}
for all $G'\Subset G$.
\end{theorem}

\begin{proof}

Since $\bfu_0\in L^2(\Omega,\mathcal F_0,\p,\mathring{W}^{1,2}(G))$ and $p\geq2$ the existence of a unique weak solution (in the sense of defintion \ref{def:weakp})
follows by the common variational approach (see for instance \cite{PrRo}) and satisfies
\begin{itemize}
\item $\bfu \in L^2(\Omega,\F,\p;L^\infty(0,T; L^2(G)))$;
\item $\bfu \in L^p(\Omega,\F,\p;L^p(0,T;  \mathring{W}^{1,p}(G)))$.
\end{itemize}
We consider a cut-off function $\eta\in C^\infty_0(G)$ and the difference quotient $\Delta_h^\gamma$ in direction $\gamma\in\left\{1,...,d\right\}$ with $|h|<\frac 12\mathrm{dist}(\support\eta,\partial\Omega)$. We apply It\^{o}'s formula to the function $f(\bfv)=\tfrac{1}{2}\|\eta\Delta_h^\gamma\bfv\|_{L^2(G)}^2$. In appropriate version it is shown in \cite{DHoV}, Prop. A.1. Although only the $L^2$-case is considered there it is straightforward to extend it to the $L^p$-setting.
This shows
\begin{align*}
\frac{1}{2}\|\eta\Delta_h^\gamma\bfu(t)\|_{L^2(G)}^2&=\frac{1}{2}\|\eta\Delta_h^\gamma\bfu_0\|_{L^2(G)}^2+\int_0^t f'(\bfu)d\bfu_\sigma+\frac{1}{2}\int_0^t f''(\bfu)d\langle\bfu\rangle_\sigma\\
&=\frac{1}{2}\|\eta\Delta_h^\gamma\bfu_0\|_{L^2(G)}^2+\int_G\int_0^t \eta^2\langle\Delta_h^\gamma\bfu\rangle_\sigma\,dx\\&+\frac{1}{2}\int_G\int_0^t\eta^2 \,d\Big\langle\int_0^{\cdot}\Delta_h^\gamma \big(\Phi(\bfu)\,d\bfW\big)\Big\rangle_\sigma\,dx
=:(I)+(II)+(III).
\end{align*}
We consider the three integrals separately.
For the second one we get
\begin{align*}
(II)&=-(II)_1-(II)_2+(II)_3,\\
(II)_1&:=\int_0^t\int_G \eta^2\Delta_h^\gamma\bfS(\nabla\bfu):\Delta_h^\gamma\nabla\bfu\,dx\,d\sigma,\\
(II)_2&:=\int_0^t\int_G \Delta_h^\gamma\bfS(\nabla\bfu):\nabla\eta^2\otimes\Delta_h^\gamma\bfu\,dx\,d\sigma,\\
(II)_3&:=\int_0^t\int_G\eta^2\Delta_h^\gamma\bfu
\cdot\Delta_h^\gamma\Big(\Phi(\bfu)\,d\bfW_\sigma\Big)\,dx.
\end{align*}
Using the assumptions for $\bfS$ in (\ref{eq:Sp}) we get
\begin{align*}
(II)_1&= \int_0^t\int_G \eta^2\int_0^1D\bfS(\nabla\bfu+sh\Delta_h^\gamma\nabla\bfu)\,ds\big
(\Delta_h^\gamma\nabla\bfu,\Delta_h^\gamma\nabla\bfu\big)\,dx\,d\sigma\\
&\geq \lambda\int_0^t\int_G \eta^2\int_0^1(1+|\nabla\bfu+sh\Delta_h^\gamma\nabla\bfu|)^{p-2}\,ds
|\Delta_h^\gamma\nabla\bfu|^2\,dx\,d\sigma\\
&\geq c\int_0^t\int_G \eta^2(1+|\nabla\bfu|+|h\Delta_h^\gamma\nabla\bfu|)^{p-2}
|\Delta_h^\gamma\nabla\bfu|^2\,dx\,d\sigma\\
&\geq c\int_0^t\int_G \eta^2
|\Delta_h^\gamma\bfF(\nabla\bfu)|^2\,dx\,d\sigma.
\end{align*}
In the second last step we used \cite{AF}, Lemma 2.1.
For the second term we obtain by similar arguments
\begin{align*}
(II)_2&\leq \,c\, \int_0^t\int_G \eta\int_0^1(1+|\nabla\bfu+sh\Delta_h^\gamma\nabla\bfu|)^{p-2}\,ds|\Delta_h^\gamma\nabla\bfu||\nabla\eta||\Delta_h^\gamma\bfu|\,dx\,d\sigma\\
&\leq \delta\int_0^t\int_G \eta^2\int_0^1(1+|\nabla\bfu+sh\Delta_h^\gamma\nabla\bfu|)^{p-2}\,ds|\Delta_h^\gamma\nabla\bfu|^2\,dx\,d\sigma\\
&+c(\delta)\int_0^t\int_G \int_0^1(1+|\nabla\bfu+sh\Delta_h^\gamma\nabla\bfu|)^{p-2}\,ds
|\nabla\eta|^2|\Delta_h^\gamma\bfu|\,dx\,d\sigma\\
&\leq c(\delta)\int_0^t\int_{\support{\eta}} (1+|\nabla\bfu|+|h\Delta_h^\gamma\nabla\bfu|)^{p-2}|\Delta_h^\gamma\bfu|^2\,dx\,d\sigma\\
&+\delta \int_0^t\int_G \eta^2
|\Delta_h^\gamma\bfF(\nabla\bfu)|^2\,dx\,d\sigma.
\end{align*}
Here we used Young's inequality for an arbitrary $\delta>0$.
Moreover, we have by (\ref{eq:phi})
\begin{align*}
(III)&=\frac{1}{2}\int_G\int_0^t \eta^2\,d\Big\langle\int_0^{\cdot}\Delta_h^\gamma \big(\Phi(\bfu)\,d\bfW\big)\Big\rangle_\sigma\,dx\\
&=\frac{1}{2}\sum_{k}\int_G\int_0^t\eta^2 \,d\Big\langle\int_0^{\cdot}\Delta_h^\gamma\big(\Phi(\bfu)\bfe_k\big) d\beta_k\Big\rangle_\sigma\,dx\\
&\leq \frac{1}{2}\sum_{k}\int_G\int_0^t\eta^2 \Big|\bigg(\int_0^1 \nabla_{\bfxi} g_k(\cdot,\bfu+sh\Delta_h^\gamma\bfu)\,ds\bigg)\Delta_h^\gamma\bfu\Big|^2\,d\sigma\,dx\\
&+\frac{1}{2}\sum_{k}\int_G\int_0^t\eta^2 \Big|\int_0^1 \partial_\gamma g_k(x+sh e_\gamma,\bfu)\,ds\Big|^2\,d\sigma\,dx\\
&\leq c\int_G\int_0^t\eta^2|\Delta_h^\gamma\bfu|^2\,d\sigma\,dx+c\int_G\int_0^t\eta^2|\bfu|^2\,d\sigma\,dx.
\end{align*}
Plugging all together and using $\E[(II)_3]=0$ we see
\begin{align*}
\E&\bigg[\int_G\eta^2|\Delta_h^\gamma\bfu(t)|^2\,dx+\int_\Q\eta^2|\Delta_h^\gamma\bfF(\nabla\bfu)|^2\,dx\,dt\bigg]\\
&\leq c\,\E\bigg[\int_G|\nabla\bfu_0|^2\,dx\bigg]+c\int_0^t\E\bigg[\int_G\eta^2\big(|\Delta_h^\gamma\bfu|^2+|\bfu|^2\big)\,dx\bigg]\,d\sigma\\
&+c\,\E\bigg[\int_0^t\int_{\support{\eta}} \eta^2(1+|\nabla\bfu|+|h\Delta_h^\gamma\nabla\bfu|)^{p-2}|\Delta_h^\gamma\bfu|^2\,dx\,d\sigma\bigg].
\end{align*}
By Gronwall's Lemma and since $\bfu_0\in L^2(\Omega,\mathcal F_0,\p; \mathring{W}^{1,2}(G))$ we end up with
\begin{align*}
\E&\bigg[\int_G\eta^2|\Delta_h^\gamma\bfu(t)|^2\,dx+\int_\Q\eta^2|\Delta_h^\gamma\bfF(\nabla\bfu)|^2\,dx\,dt\bigg]\\&\leq c(\eta)\bigg(1+\E\bigg[\int_0^t\int_{\support{\eta}}(1+|\nabla\bfu|+|h\Delta_h^\gamma\nabla\bfu|)^{p-2}|\Delta_h^\gamma\bfu|^2\,dx\,d\sigma\bigg]\bigg).
\end{align*}
Here we also took into account $\bfu\in L^2(\Omega\times Q)$.
If $p>2$ we gain by Young's inequality for the exponents $\frac{p}{2}$ and $\frac{p}{p-2}$\footnote{This step is trivial if $p=2$.}
\begin{align*}
(RHS)&\leq c(\eta)\bigg(1+\int_0^t\int_{G} |\nabla\bfu|^{p}\,dx\,d\sigma+\int_0^t\int_{\support{\eta}} |h\Delta_h^\gamma\nabla\bfu|^{p}\,dx\,d\sigma\bigg)\\
&\leq c(\eta)\bigg(1+\int_0^t\int_{G} |\nabla\bfu|^{p}\,dx\,d\sigma\bigg)
\end{align*}
which is bounded as well (independent of $h$). This means we have shown
\begin{align*}
\E&\bigg[\int_G\eta^2|\Delta_h^\gamma\bfu(t)|^2\,dx+\int_\Q\eta^2|\Delta_h^\gamma\bfF(\nabla\bfu)|^2\,dx\,dt\bigg]\leq c(\eta).
\end{align*}
Now we want to interchange supremum and expectation value. Applying similar arguments as before we obtain
\begin{align}\label{eq:GG}
\begin{aligned}
\E\bigg[\sup_{(0,T)}&\int_G\eta^2|\Delta_h^\gamma\bfu(t)|^2\,dx\bigg]+\E\bigg[\int_\Q\eta^2|\Delta_h^\gamma\bfF(\nabla\bfu)|^2\,dx\,dt\bigg]\\
&\leq c(\eta)+c\,\E\bigg[\sup_{(0,T)}|(II)_3|\bigg].
\end{aligned}
\end{align}
Using the assumptions on $\bfW$ (see (\ref{eq:phi})) we see
\begin{align*}
(II)_3&=\int_G\int_0^t\eta^2\Delta_h^\gamma\bfu\cdot\Delta_h^\gamma\Big(\Phi(\bfu)\bfe_k\,d\beta_k(\sigma)\Big)\,dx\\
&=\sum_k\int_G\int_0^t\eta^2\Delta_h^\gamma\bfu\cdot\Delta_h^\gamma\Big(g_k(\cdot,\bfu)\,d\beta_k(\sigma)\Big)\,dx\\
&=\sum_k\int_G\int_0^t\eta^2\bigg(\int_0^1 \nabla_{\bfxi} g_k(\cdot,\bfu+sh\Delta_h^\gamma\bfu)\,ds\bigg)
(\Delta_h^\gamma\bfu,\Delta_h^\gamma\bfu)\,d\beta_k(\sigma)\,dx\\
&+\sum_k\int_G\int_0^t\eta^2\Delta_h^\gamma\bfu\cdot\bigg(\int_0^1 \partial_\gamma g_k(x+she_\gamma,\bfu)\,ds\bigg)\,d\beta_k(\sigma)\,dx\\
&=\int_G\int_0^t\eta^2\mathcal G^\bfxi
(\Delta_h^\gamma\bfu,\Delta_h^\gamma\bfu)\,d\beta_k(\sigma)\,dx\\
&+\int_G\int_0^t\eta^2\mathcal G^x(\bfu)\cdot\Delta_h^\gamma\bfu\,d\beta_k(\sigma)\,dx\\
&=:(II)_3^1+(II)_3^2
\end{align*}
where we abbreviated
\begin{align*}
\quad \mathcal G^\bfxi&:=\sum_k\mathcal G^\bfxi_k:=\sum_k\bigg(\int_0^1 \nabla_{\bfxi} g_k(\cdot,\bfu+sh\Delta_h^\gamma\bfu)\,ds\bigg),\\
\mathcal G^x(\bfu)&:=\sum_k\mathcal G^x_k(\bfu):=\sum_k\int_0^1 \partial_\gamma g_k(x+she_\gamma,\bfu)\,ds.
\end{align*}
On account of assumption (\ref{eq:phi}) Burkholder-Davis-Gundy inequality and Young's inequality imply for arbitrary $\delta>0$
\begin{align*}
\E\bigg[\sup_{t\in(0,T)}|(II)^1_3|\bigg]&\leq \E\bigg[\sup_{t\in(0,T)}\Big|\int_0^t\sum_k\int_G\eta^2\mathcal G^\bfxi_k
(\Delta_h^\gamma\bfu,\Delta_h^\gamma\bfu)\,dx\,d\beta_k(\sigma)\Big|\bigg]\\
&\leq c\,\E\bigg[\int_0^T\bigg(\int_G \eta^2 \mathcal G^\bfxi
(\Delta_h^\gamma\bfu,\Delta_h^\gamma\bfu)\,dx\bigg)^2\,dt\bigg]^{\frac 12}\\
&\leq c\,\E\bigg[\bigg(\int_0^T\bigg(\int_G\eta^2
|\Delta_h^\gamma\bfu|^2\,dx\bigg)^2\,dt\bigg]^{\frac 12}\\
&\leq \delta\,\E\bigg[\sup_{(0,T)}\int_G\eta^2
|\Delta_h^\gamma\bfu|^2\,dx\bigg]+ c(\delta)\,\E\bigg[\int_\Q\eta^2
|\Delta_h^\gamma\bfu|^2\,dx\,dt\bigg].
\end{align*}
By similar arguments we gain
\begin{align*}
\E\bigg[\sup_{t\in(0,T)}|(II)^2_3|\bigg]
&\leq c\,\E\bigg[\int_0^T\bigg(\int_G \eta^2 \mathcal G^x(\bfu)
\cdot\Delta_h^\gamma\bfu\,dx\bigg)^2\,dt\bigg]^{\frac 12}\\
&\leq c\,\E\bigg[\bigg(\int_0^T\bigg(\int_G\eta^2
|\Delta_h^\gamma\bfu||\bfu|\,dx\bigg)^2\,dt\bigg]^{\frac 12}\\
&\leq c\,\E\bigg[\sup_{(0,T)}\int_G\eta^2
|\bfu|^2\,dx\bigg]+ c\,\E\bigg[\int_\Q\eta^2
|\Delta_h^\gamma\bfu|^2\,dx\,dt\bigg].
\end{align*}
Combining this with (\ref{eq:GG}), using $\bfu\in L^2(\Omega,\F,\p;L^\infty(0,T;L^2(G)))$ and choosing $\delta$ sufficiently small shows
\begin{align}\label{eq:E_sup}
\E\bigg[\sup_{(0,T)}\int_G\eta^2|\Delta_h^\gamma\bfu(t)|^2\,dx\bigg]+\E\bigg[\int_\Q\eta^2|\Delta_h^\gamma\bfF(\nabla\bfu)|^2\,dx\,dt\bigg]
\leq c(\eta).
\end{align}

This finally proves the claim (see \cite{BeFl}, section 3.2, for difference quotients and differentiability in the stochastic setting).
\end{proof}

\section{The subquadratic case: $p<2$}
Throughout this section we study problems of the type (\ref{eq:}) with (\ref{eq:Sp}) and $p\leq2$.
We add the Laplacian to the main part in order to get a problem with quadratic growth. Let $\bfu^\epsilon$ be the solution to
\begin{align}\label{eq:strongsub}
\begin{cases}
d\bfu^\epsilon&=\Div\big(\bfS(\nabla \bfu^\epsilon)\big)dt+\epsilon\Delta\bfu dt+\Phi(\bfu^\epsilon)d\bfW_t,\\
\bfu(0)&=\bfu_0.
\end{cases}
\end{align}
From Theorem \ref{thm:2.3} we know that the solution has the following properties
\begin{itemize}
\item $\bfu^\epsilon \in L^2(\Omega,\F,\p;L^\infty(0,T;L^2(G)))$;
\item $\nabla \bfu^\epsilon \in L^2(\Omega,\F,\p;L^2(0,T;W^{1,2}_{loc}(G)))$.
\end{itemize}
We will prove the following a priori estimates which are uniform in $\epsilon$:
\begin{align}\label{eq:aprisub}
\begin{aligned}
\E\bigg[\sup_{t\in(0,T)}&\int_G |\bfu^\epsilon(t)|^2\,dx+\int_\mathcal{Q} |\nabla \bfu^\epsilon|^p\,dx\,dt+\epsilon\int_\mathcal{Q} |\nabla \bfu^\epsilon|^2\,dx\,dt\bigg]\\&\leq c\bigg(1+\E\bigg[\int_G |\bfu_0|^2\,dx\bigg]\bigg).
\end{aligned}
\end{align}
We apply It\^{o}'s formula to the function $f(\bfv)=\tfrac{1}{2}\|\bfv\|_{L^2(G)}^2$ which shows
\begin{align}
\frac{1}{2}\|\bfu^\epsilon(t)\|_{L^2(G)}^2&=\frac{1}{2}\|\bfu_0\|_{L^2(G)}^2+\int_0^t f'(\bfu^\epsilon)d\bfu^\epsilon_\sigma+\frac{1}{2}\int_0^t f''(\bfu^\epsilon)\,d\langle\bfu^\epsilon\rangle_\sigma\nonumber\\
&=\frac{1}{2}\|\bfu_0\|_{L^2(G)}^2-\epsilon\int_G\int_0^t|\nabla\bfu^\epsilon|^2\,dx\,d\sigma-\int_{G}\int_0^t\bfS(\nabla \bfu^\epsilon):\nabla\bfu^\epsilon\,dx\,d\sigma\nonumber\\
&+\int_G\int_0^t\bfu^\epsilon\cdot\Phi(\bfu^\epsilon)\,d\bfW_\sigma\,dx
+\int_G\int_0^td\Big\langle\int_0^{\cdot}\Phi(\bfu^\epsilon)\, d\bfW\Big\rangle_\sigma\,dx.\label{eq:cuN}
\end{align}
Now we can follow, building expectations and using (\ref{eq:Sp}), that
\begin{align*}
\E\bigg[\int_G &|\bfu^\epsilon(t)|^2\,dx+\epsilon\int_0^t\int_G |\nabla \bfu^\epsilon|^2\,dx\,d\sigma+\int_0^t\int_G |\nabla \bfu^\epsilon|^p\,dx\,d\sigma\\&\leq c\Big(\E\big[1+\|\bfu_0\|^2_{L^2(G)}\big]+\E\big[J_1(t)\big]+\E\big[J_2(t)\big]\Big).
\end{align*}
Here we abbreviated
\begin{align*}
J_1(t)&=\int_G\int_0^t\bfu^\epsilon\cdot\Phi(\bfu^\epsilon)\,d\bfW_\sigma\,dx,\\
J_2(t)&=\int_G\int_0^td\Big\langle\int_0^{\cdot}\Phi(\bfu^\epsilon)\, d\bfW\Big\rangle_\sigma\,dx.
\end{align*}
Using (\ref{eq:phi}) we gain
\begin{align*}
\E[J_2]&= \E\bigg[\int_0^t\sum_{i=1}^\infty\int_G |\Phi(\bfu^\epsilon)\bfe_i|^2\,dx\,d\sigma\bigg]\\
&= \E\bigg[\int_0^t\sum_{i=1}^\infty\int_G |g_i(\cdot,\bfu^\epsilon)|^2\,dx\,d\sigma\bigg]\\
&\leq\,c\, \E\bigg[1+\int_0^t\int_G|\bfu^\epsilon|^2\,dx\,d\sigma\bigg].
\end{align*}
Clearly, we have $\E[J_1]=0$. So
interchanging the time-integral and the expectation value and applying Gronwall's Lemma leads to
\begin{align}\label{eq:4.4}
\begin{aligned}
\sup_{t\in(0,T)}\E&\bigg[\int_G |\bfu^\epsilon(t)|^2\,dx\bigg]+\epsilon\E\bigg[\int_\mathcal{Q} |\nabla \bfu^\epsilon|^p\,dx\,dt\bigg]+\E\bigg[\int_\mathcal{Q} |\nabla \bfu^\epsilon|^p\,dx\,dt\bigg]\\&\leq \,c\,\E\bigg[1+\int_G |\bfu_0|^2\,dx\bigg].
\end{aligned}
\end{align}
A similar observation shows
\begin{align}\label{eq:4.4'}
\begin{aligned}
\E\bigg[\sup_{t\in(0,T)}\int_G |\bfu^\epsilon(t)|^2\,dx\bigg]&\leq \,c\,\E\bigg[1+\int_G |\bfu_0|^2\,dx+\int_0^T\int_G |\bfu^\epsilon|^2\,dx\,d\sigma\bigg]\\&+\,c\,\E\bigg[\sup_{t\in(0,T)}|J_1(t)|\bigg].
\end{aligned}
\end{align}
On account of the Burkholder-Davis-Gundy inequality, (\ref{eq:phi}) and Young's inequality we obtain for arbitrary $\kappa>0$
\begin{align*}
\E\bigg[\sup_{t\in(0,T)}|J_1(t)|\bigg]&=\E\bigg[\sup_{t\in(0,T)}\Big|\int_0^t\int_G\bfu^\epsilon\Phi(\bfu^\epsilon)\,dx\,d\bfW_\sigma\Big|\bigg]\\
&=\E\bigg[\sup_{t\in(0,T)}\Big|\int_0^t\sum_i\int_G\bfu^\epsilon\cdot g_i(\cdot,\bfu^\epsilon)\,dx\,d\beta_i(\sigma)\Big|\bigg]\\
&\leq c\,\E\bigg[\int_0^T\sum_i\bigg(\int_G|\bfu^\epsilon|g_i(\cdot,\bfu^\epsilon)\,dx\bigg)^2\,dt\bigg]^{\frac{1}{2}}\\
&\leq c\,\E\bigg[1+\bigg(\int_0^T\bigg(\int_G|\bfu^\epsilon|^2\,dx\bigg)^2\,d\sigma\bigg]^{\frac{1}{2}}\\
&\leq \kappa\E\bigg[\sup_{t\in(0,T)}\int_G|\bfu^\epsilon|^2\,dx\bigg]+c(\kappa)\E\bigg[1+\int_0^T\int_G|\bfu^\epsilon|^2\,dx\,d\sigma\bigg]
\end{align*}
Inserting this in (\ref{eq:4.4'}), choosing $\kappa$ small enough and using (\ref{eq:4.4}) proves (\ref{eq:aprisub}).\\
After passing to a (not relabeled) subsequence we have for a certain function $\bfu$
\begin{equation}
\label{eq:conv_sub}
\begin{aligned}
\bfu^\epsilon&\rightharpoondown \bfu \quad\text{in}\quad L^p(\Omega,\F,\p;L^p(\Q)),\\
\bfu^\epsilon&\rightharpoondown \bfu \quad\text{in}\quad L^2(\Omega,\F,\p;L^r(0,T;L^2(G)))\quad\forall r<\infty,\\
\nabla\bfu^\epsilon&\rightharpoondown \nabla\bfu \quad\text{in}\quad L^p(\Omega,\F,\p;L^p(\Q)),\\
\epsilon\nabla\bfu^\epsilon &\rightarrow 0 \quad\text{in}\quad L^2(\Omega,\F,\p;L^2(\Q)).
\end{aligned}
\end{equation}
\begin{theorem}[Regularity]\label{thm:sub2}
$\left.\right.$\\
Assume (\ref{eq:Sp}) with $p\leq2$, (\ref{eq:phi}) and $\bfu_0\in L^2(\Omega,\mathcal F_0,\p,\mathring{W}^{1,2}(G))$.
Then there is a unique weak solution $\bfu$ to (\ref{eq:}) which is a strong solution and satisfies
\begin{align*}
\E\bigg[\sup_{t\in(0,T)}\int_{G'}|\nabla\bfu(t)|^2\,dx+\int_0^T\int_{G'}|\nabla\bfF(\nabla\bfu)|^2\,dx\,dt\bigg]<\infty
\end{align*}
for all $G'\Subset G$, where $\bfF(\bfxi)=(1+|\bfxi|)^{\frac{p-2}{2}}\bfxi$.
\end{theorem}
\begin{remark}
In cthe ase $1<p<\frac{2d}{d+2}$ even the existence of a weak solution is not contained in literature. In this
case no Gelfand triple is available and hence the general results for evolutionary SPDEs based on the variational approach (see for instance \cite[Thm. 4.2.4]{PrRo}) do not hold. The uniqueness is again classical and follows from the monotonicity of the coefficients.
\end{remark}
\begin{proof}
From the proof of Theorem \ref{thm:2.3} we can quote (recall (\ref{eq:aprisub}))
\begin{align*}
\E&\bigg[\int_G\eta^2|\Delta_h^\gamma\bfu^\epsilon(t)|^2\,dx+\int_\Q\eta^2|\Delta_h^\gamma\bfF(\nabla\bfu^\epsilon)|^2\,dx\,dt\bigg]\\&\leq c(\eta)\bigg(1+\E\bigg[\int_0^t\int_{\support{\eta}}(1+|\nabla\bfu^\epsilon|+|h\Delta_h^\gamma\nabla\bfu^\epsilon|)^{p-2}|\Delta_h^\gamma\bfu^\epsilon|^2\,dx\,d\sigma\bigg]\bigg)\\
&+c(\eta)\epsilon\,\E\bigg[\int_0^t\int_{\support{\eta}}|\Delta_h^\gamma\bfu^\epsilon|^2\,dx\,d\sigma\bigg]\bigg)
\end{align*}
since the arguments up to this step also work for $p\leq2$. All involved quantities have weak derivatives so we can go to the limit $h\rightarrow0$. This shows by (\ref{eq:aprisub})
\begin{align*}
\E&\bigg[\int_G\eta^2|\nabla\bfu^\epsilon(t)|^2\,dx+\int_\Q\eta^2|\nabla\bfF(\nabla\bfu^\epsilon)|^2\,dx\,dt\bigg]\\&\leq c(\eta)\bigg(1+\E\bigg[\int_0^t\int_{G}(1+|\nabla\bfu^\epsilon|)^{p-2}|\nabla\bfu^\epsilon|^2\,dx\,d\sigma\bigg]\bigg)\\
&+c(\eta)\epsilon\,\E\bigg[\int_0^t\int_{G}|\nabla\bfu^\epsilon|^2\,dx\,d\sigma\bigg]\bigg).\\
&\leq c(\eta)\bigg(1+\E\bigg[\int_0^t\int_{G}(1+|\nabla\bfu^\epsilon|)^{p-2}|\nabla\bfu^\epsilon|^2\,dx\,d\sigma\bigg]\bigg)\\
&\leq c(\eta)\bigg(1+\E\bigg[\int_0^t\int_{G}|\nabla\bfu^\epsilon|^p\,dx\,d\sigma\bigg]\bigg).
\end{align*}
Using similar arguments as in the last section we can interchange supremum and integral and conclude
\begin{align}\label{eq:estpll2}
\begin{aligned}
\E&\bigg[\sup_{t\in(0,T)}\int_G\eta^2|\nabla\bfu^\epsilon(t)|^2\,dx+\int_\Q\eta^2|\nabla\bfF(\nabla\bfu^\epsilon)|^2\,dx\,dt\bigg]\\
&\leq c(\eta)\bigg(1+\E\bigg[\int_\Q|\nabla\bfu^\epsilon|^p\,dx\,d\sigma\bigg]\bigg)\leq c(\eta).
\end{aligned}
\end{align}
Now we have to go to the limit in the equation. We get
\begin{align}\label{eq:conv_eps}
\begin{aligned}
\bfS(\nabla\bfu^\epsilon)&\rightharpoondown: \tilde{\bfS} \quad\text{in}\quad L^{p'}(\Omega,\F,\p;L^{p'}(\Q)),\\
\Phi(\bfu^\epsilon)&\rightharpoondown: \tilde{\Phi} \quad\text{in}\quad L^2(\Omega,\F,\p;L^2(0,T;L_2(U,L^2(G)^D))).
\end{aligned}
\end{align}
One can now pass to the limit in the
equation to obtain the corresponding equation for u with $\tilde{\bfS}$ and $\tilde{\Phi}$ instead of $\bfS(\nabla\bfu)$ and $\Phi(\bfu)$,
respectively. The passage to the limit in the stochastic integral is justified since the
mapping
\begin{align*}
L^2(\Omega,\F,\p;L^2(0,T;L_2(U;L^2(G))))&\rightarrow L^2(\Omega,\F,\p;L^2(0,T;L^2(G))),\\
\bfphi&\mapsto \int_0^t\bfphi\,d\bfW_\sigma,
\end{align*}
is continuous hence weakly continuous.
We have to show that $\tilde{\bfS}=\bfS(\nabla\bfu)$ and $\tilde{\Phi}=\Phi(\bfu)$ hold. Subtracting the formula for $\|\bfu^\epsilon\|^2_{L^2(G)}$ and  $\|\bfu\|^2_{L^2(G)}$ (see (\ref{eq:cuN})) shows
\begin{align*}
&\frac{1}{2}\,\E\bigg[\int_G|\bfu^\epsilon(T)- \bfu(T)|^2\,dx\bigg]\\
+&\E\bigg[\int_G\int_0^T\big(\bfS(\nabla\bfu^\epsilon)-\bfS(\nabla\bfu)\big):\nabla\big(\bfu^\epsilon-\bfu\big)\,dx\,d\sigma\bigg]+\epsilon\,\E\bigg[\int_0^T\int_G|\nabla\bfu^\epsilon|^2\,dx\,d\sigma\bigg]\\
&=\E\bigg[-\int_G\bfu^\epsilon(T)\cdot \bfu(T)\,dx\bigg]\\
&+\E\bigg[\int_G\int_0^T\big(\tilde{\bfS}-\bfS(\nabla\bfu^\epsilon)\big):\nabla\bfu\,dx\,d\sigma-\int_G\int_0^T\bfS(\nabla\bfu):\nabla\big(\bfu^\epsilon-\bfu\big)\,dx\,d\sigma\bigg]\\
&+\E\bigg[\int_G\int_0^T\Big(\bfu^\epsilon\cdot\Phi(\bfu^\epsilon) d\bfW_\sigma
-\bfu\cdot\tilde{\Phi} d\bfW_\sigma\Big)\,dx\bigg]\\
&+\E\bigg[\int_G\int_0^Td\Big(\Big\langle\int_0^{\cdot}\Phi(\bfu^\epsilon) d\bfW\Big\rangle_\sigma
-\Big\langle\int_0^{\cdot}\tilde{\Phi} d\bfW\Big\rangle_\sigma\Big)\,dx\bigg].
\end{align*}
By (\ref{eq:conv_sub}) $\bfu^\epsilon(T)$ is bounded in $L^2(\Omega\times G,\p\otimes\mathcal L^d)$. Which gives $\bfu^\epsilon(T)\rightharpoondown \bfu(T)$ in the same space at least for a subsequence (note that both are weakly continuous in $L^2(\Omega\times G,\p\otimes\mathcal L^d)$ with respect to $t$ which can be shown by the equations).
Letting $\epsilon\rightarrow\infty$ shows for a subsequence using (\ref{eq:conv_sub}) and (\ref{eq:conv_eps})
\begin{align*}
\lim_\epsilon\E&\bigg[\int_G|\bfu^\epsilon(T)- \bfu(T)|^2\,dx+\int_G\int_0^T\big(\bfS(\nabla\bfu^\epsilon)-\bfS(\nabla\bfu)\big):\nabla\big(\bfu^\epsilon-\bfu\big)\,dx\,d\sigma\bigg]\\
&\leq \lim_\epsilon\E\bigg[\int_G\int_0^Td\Big(\Big\langle\int_0^{\cdot}\Phi(\bfu^\epsilon) d\bfW\Big\rangle_\sigma
-\Big\langle\int_0^{\cdot}\tilde{\Phi} d\bfW\Big\rangle_\sigma\Big)\,dx\bigg].
\end{align*}
Following essential ideas of \cite{ChCh} (section 6) the last integral $\tilde{T}$ can be written as
\begin{align*}
\tilde{T}&=\sum_i \E\bigg[\int_G\int_0^T|\Phi(\bfu^\epsilon)\bfe_i|^2\,dx\,d\sigma\bigg]-\sum_i \E\bigg[\int_G\int_0^T|\tilde{\Phi}\bfe_i|^2\,dx\,d\sigma\bigg]\\
&=\E\bigg[\int_0^T\|\Phi(\bfu^\epsilon)\|^2_{L_2(U,L^2(G))}\,dt\bigg]-\E\bigg[\int_0^T\|\tilde{\Phi}\|^2_{L_2(U,L^2(G))}\,dt\bigg]\\
&=\E\bigg[\int_0^T\|\Phi(\bfu^\epsilon)-\tilde{\Phi}\|^2_{L_2(U,L^2(G))}\,dt\bigg]+2\,\E\bigg[\int_0^T\Big\langle\Phi(\bfu^\epsilon),\tilde{\Phi}\Big\rangle_{L_2(U,L^2(G))}\,dt\bigg]\\&-2\,\E\bigg[\int_0^T\|\tilde{\Phi}\|^2_{L_2(U,L^2(G))}\,dt\bigg].
\end{align*}
On account of (\ref{eq:conv_eps}) for $\epsilon\rightarrow 0$ we only have to consider the first term which can be written as
\begin{align*}
\E\bigg[\int_0^T\|\Phi(\bfu^\epsilon)-\tilde{\Phi}\|^2_{L_2(U,L^2(G))}\,dt\bigg]&=\E\bigg[\int_0^T\|\Phi(\bfu^\epsilon)-\Phi(\bfu)\|^2_{L_2(U,L^2(G))}\,dt\bigg]\\&-\E\bigg[\int_0^T\|\Phi(\bfu)-\tilde{\Phi}\|^2_{L_2(U,L^2(G))}\,dt\bigg]\\
&+2\,\E\bigg[\int_0^T\Big\langle\Phi(\bfu^\epsilon)-\tilde{\Phi},\Phi(\bfu)-\tilde{\Phi}\Big\rangle_{L_2(U,L^2(G))}\,dt\bigg]
\end{align*}
Using again (\ref{eq:conv_eps}) and also (\ref{eq:phi}) implies
\begin{align*}
\lim_\epsilon \tilde{T}&\leq \lim_\epsilon\E\bigg[\int_0^T\|\Phi(\bfu^\epsilon)-\Phi(\bfu)\|^2_{L_2(U,L^2(G))}\,dt\bigg]\\
&\leq \,c\,\lim_\epsilon\E\bigg[\int_0^T\int_G|\bfu^\epsilon-\bfu|^2\,dx\,dt\bigg].
\end{align*}
We finally gain on account of Grownwall's lemma after interchanging expectation and integral
\begin{align*}
\E\bigg[\int_G\int_0^T\big(\bfS(\nabla\bfu^\epsilon)-\bfS(\nabla\bfu)\big):\nabla\big(\bfu^\epsilon-\bfu\big)\,dx\,d\sigma\bigg]=0.
\end{align*}
From this we deduce, by monotonicity of $\bfS$ that
\begin{align*}
\nabla\bfu^\epsilon\longrightarrow \nabla\bfu\quad\mathbb P\otimes \mathcal L^{d+1}-a.e.
\end{align*}
This means we have shown $\tilde{\bfS}=\bfS(\nabla\bfu)$.
Now we combine the uniform $L^p$-estimates for $\nabla\bfu^\epsilon$ with Vitali's Theorem to get
\begin{align}\label{eq:compactnablaN}
\nabla\bfu^\epsilon\longrightarrow \nabla\bfu\quad\text{in}\quad L^q(\Omega\times(0,T)\times G;\p\otimes\mathcal L^{d+1})\quad\text{for all } q<p.
\end{align}
Of course this also means compactness of $\bfu^\epsilon$ in the same space (we have zero traces). Therefore, we gain $\tilde{\Phi}=\Phi(\bfu)$.
Now we can pass to the limit in the approximated equation and finish the proof of Theorem \ref{thm:sub2}.
\end{proof}

\section{Uhlenbeck-structure}
In order to get better results we assume Uhlenbeck structure for the non-linear tensor $\bfS$. If $D\geq2$ we suppose
\begin{align}\label{eq:nu}
\bfS(\bfxi)=\nu(|\bfxi|)\bfxi
\end{align}
for a $C^1$-function $\nu:[0,\infty)\rightarrow[0,\infty)$.
\begin{theorem}[Higher integrability]\label{thm:4.3} 
$\left.\right.$\\
Assume (\ref{eq:Sp}), (\ref{eq:nu}), (\ref{eq:phi}) and $\bfu_0\in L^q(\Omega,\F_0,\p,\mathring{W}^{1,q}(G))$ for all $q<\infty$. 
If $p>2-\frac{4}{d}$ then the solution $\bfu$ to (\ref{eq:}) satisfies
\begin{align*}
\E\bigg[\int_0^T\int_{G'}|\nabla\bfu|^q\,dx\,dt\bigg]<\infty
\end{align*}
for all $G'\Subset G$ and all $q<\infty$.
\end{theorem}

Since we now assume higher moments for the initial data we gain higher moments for the solution as well.
\begin{lemma}\label{lems:hm} 
$\left.\right.$\\
Under the assumptions of Theorem \ref{thm:4.3} we have
\begin{align*}
\E\bigg[\sup_{t\in(0,T)}\int_{G}|\bfu(t)|^2\,dx+\int_0^T\int_{G}|\nabla\bfu|^p\,dx\,dt\bigg]^q<\infty
\end{align*}
for all $q<\infty$.
\end{lemma}
\begin{proof}
Due to the regularity results from Theorem \ref{thm:2.3} and \ref{thm:sub2} we have a strong solution and It\^{o}'s formula can be directly applied to the funtion $f(\bfv)=\frac{1}{2}\|\bfv\|_{L^2(G)}^2$. Using the growth condition on $\bfS$ from (\ref{eq:Sp}), taking the supremum and the $q$-th power of both sides of the equation and applying expectations shows
\begin{align*}
\E\bigg[\sup_{(0,T)}\int_G &|\bfu(t)|^2\,dx+\int_0^T\int_G |\nabla \bfu|^p\,dx\,d\sigma\bigg]^q\\&\leq c\bigg(\E\bigg[1+\int_G|\bfu_0|^{2q}\,dx\bigg]+\E\bigg[\sup_{(0,T)}|J_1(t)|\bigg]^q+\E\bigg[\sup_{(0,T)}|J_2(t)|\bigg]^q\bigg),\\
J_1(t)&=\int_G\int_0^t\bfu\cdot\Phi(\bfu)\,d\bfW_\sigma\,dx,\\
J_2(t)&=\int_G\int_0^td\Big\langle\int_0^\cdot\Phi(\bfu) \,d\bfW\Big\rangle_\sigma\,dx.
\end{align*}
Using (\ref{eq:phi}) we gain
\begin{align*}
\E\bigg[\sup_{t\in(0,T)}|J_2(t)|\bigg]^q&= \E\bigg[\sup_{t\in(0,T)}\int_0^t\sum_{i=1}^\infty\int_G |\Phi(\bfu)\bfe_i|^2\,dx\,d\sigma\bigg]^q\\
&\leq  \E\bigg[\int_0^T\sum_{i=1}^\infty\int_G |g_i(\cdot,\bfu)|^2\,dx\,d\sigma\bigg]^q\\
&\leq \,c\,\E\bigg[1+\int_0^T\int_G|\bfu|^2\,dx\,d\sigma\bigg]^q.
\end{align*}
On account of the Burkholder-Davis-Gundi inequality, (\ref{eq:phi}) and Young's inequality we obtain for arbitrary $\epsilon>0$
\begin{align*}
\E\bigg[\sup_{t\in(0,T)}|J_1(t)|\bigg]^q
&=\E\bigg[\sup_{t\in(0,T)}\Big|\sum_i\int_0^t\int_G\bfu\cdot g_i(\cdot,\bfu)\,dx\,d\beta_i(\sigma)\Big|\bigg]^q\\
&\leq c\,\E\bigg[\int_0^T\sum_i\bigg(\int_G\bfu\cdot g_i(\cdot,\bfu)\,dx\bigg)^2\,dt\bigg]^{\frac{q}{2}}\\
&\leq c\,\E\bigg[1+\bigg(\int_0^T\bigg(\int_G|\bfu|^2\,dx\bigg)^2\,d\sigma\bigg]^{\frac{q}{2}}\\
&\leq \epsilon\E\bigg[\sup_{t\in(0,T)}\int_G|\bfu|^2\,dx\bigg]^q+c(\epsilon)\E\bigg[\int_0^T\int_G|\bfu|^2\,dx\,d\sigma\bigg]^q.
\end{align*}
Choosing $\epsilon$ small enough and using Gronwall's lemma proves the claim.\\
Since the calculations above are not well-defined a priori one can work with a quadratic approximation for the function $z\mapsto z^q$.
\end{proof}

Before we begin with the proof of Theorem \ref{thm:4.3} which is based on the Moser iteration (see \cite{GT} for a nice presentation in the easier elliptic case) we need some preparations. The basic idea is estimating higher powers of $|\nabla\bfu|$ by lower powers and iterate this. Therefore, we define $$h(s):=\int_0^s(1+\theta)^{\alpha}\theta\,d\theta,\quad \alpha\geq0,$$ 
which behaves like $s^{\alpha+2}$ for large $s$.
Unfortunately we cannot work directly with $h$, we need an approximation $h_L$ which grows quadratically and converges to $h$. We follow the approach in \cite{BiFu} and define for $L\gg1$
\begin{align}\label{eq:hL}
\begin{aligned}
h_{L}(s)&:=\int_{0}^ s\tau g_{L}(\tau)\,d\tau,\\
g_{L}(\tau)&:=g(0)+\int_{0}^\tau\psi(\theta)g'(\theta)\,d\theta,\\
g(\theta)&:=\frac{h'(\theta)}{\theta}.
\end{aligned}
\end{align}
Here $\psi\in C^1([0,\infty))$ denotes a cut-off function with the properties $0\leq\psi\leq1$, $\psi'\leq0$, $|\psi'|\leq c/L$, $\psi\equiv1$ on $[0,3L/2]$ and $\psi\equiv0$ on $[2L,\infty)$. For the function $h_L$ we obtain the following properties (see \cite{Br}, Lemma 2.1, and \cite{Br2}, section 2)
\begin{lemma}
\label{lems41}
For the sequence $(h_{L})$ we have:
\begin{enumerate}
\item $h_{L}\in C^2[0,\infty)$, $h_{L}(s)=h(s)$ for all $t\leq 3L/2$ and
\begin{align*}
\lim_{L\rightarrow\infty}h_{L}(s)=h(s)\,\,\text{ for all }s\geq0;\qquad \qquad\qquad
\end{align*}
\item $h_{L}\leq h$, $g_{L}\leq g$ and $h_{L}''\leq c(L)$ on $[0,\infty)$;
\item It holds
\begin{align*}
\,\frac{h_{L}'(s)}{s}\leq h_{L}''(s)\leq c(\alpha+1)\,\frac{h_{L}'(s)}{s}
\end{align*}
and $h_{L}'(s)s\leq ch_L(s)$ uniformly in $L$.
\item We have for all $s,t\geq0$ uniformly in $L$
\begin{align*}
\frac{h_L'(s)}{s}t^2\leq \,c(\alpha)\big(1+h_L(s)+h_L(t)t^2\big).
\end{align*}
\end{enumerate}
\end{lemma}
With this preparations the following calculations are well-defined by Theorem \ref{thm:2.3} and Theorem \ref{thm:sub2}. 

\begin{lemma}\label{lems:alphau} 
$\left.\right.$\\
Under the assumptions of Theorem \ref{thm:4.3} we have
\begin{align*}
\E\bigg[\sup_{t\in(0,T)}\int_{G}|\bfu(t)|^q\,dx\bigg]<\infty
\end{align*}
for all $q<\infty$.
\end{lemma}
\begin{proof}
We apply It\^{o}'s formula to the function $$f_L(\bfv):=\int_GH_L(\bfv)\,dx:=\int_Gh_L(|\bfv|)\,dx,$$ 
where $h_L$ is defined in (\ref{eq:hL}) and set $\alpha=q-2$. We obtain
\begin{align*}
\int_G&h_L(|\bfu|)\,dx\\&=\int_G\eta^2h_L(|\bfu_0|)\,dx+\int_0^t f'_L(\bfu)d\bfu_\sigma+\frac{1}{2}\int_0^t f''_L(\bfu)\,d\langle\bfu_\sigma\rangle_\sigma\\
&=\int_G\eta^2h_L(|\bfu_0|)\,dx+\int_G \int_0^t DH_L(\bfu)\cdot d\bfu_\sigma\,dx\\&+\int_G\int_0^t D^2H_L(\bfu) \,d\Big\langle\int_0^{\cdot}\Phi(\bfu)\,d\bfW \Big\rangle_\sigma\,dx\\
&=:(I)_q+(II)_q+(III)_q.
\end{align*}
We consider the three integrals separately and decompose the second one into
\begin{align*}
(II)_q&=-(II)^1_q-(II)^2_\alpha+(II)^3_q,\\
(II)_q^1&:=\int_0^t\int_G\tfrac{h_L'(|\bfu|)}{|\bfu|}\bfS(\nabla\bfu):\nabla\bfu\,dx,\\
(II)_q^2&:=\int_0^t\int_G \bfS(\nabla\bfu):\nabla\tfrac{h_L'(|\bfu|)}{|\bfu|}\otimes\bfu\,dx\,d\sigma,\\
(II)_q^3&:=\int_0^t\int_G\tfrac{h_L'(|\bfu|)}{|\bfu|}\bfu
\cdot\Phi(\bfu)\,d\bfW_\sigma\,dx.
\end{align*}
Using the Uhlenbeck structure (\ref{eq:nu}) and Lemma \ref{lems41} c) we gain
\begin{align*}
(II)_q^1&=\int_0^t\int_G\tfrac{h_L'(|\bfu|)}{|\nabla\bfu|}\nu(|\nabla\bfu|)|\nabla\bfu|^2\,dx\geq0,\\
(II)_q^2&=\int_0^t\int_G \nu(|\nabla\bfu|)\nabla\bfu:\tfrac{h_L''(|\bfu|)|\bfu|-h_L'(|\bfu|)}{|\bfu|^2}\nabla|\bfu|\otimes\bfu\,dx\,d\sigma\\
&=\frac{1}{4}\int_0^t\int_G \nu(|\nabla\bfu|)\tfrac{h_L''(|\bfu|)|\bfu|-h_L'(|\bfu|)}{|\bfu|^3}\big|\nabla|\nabla\bfu|^2\big|^2\,dx\,d\sigma\geq0.
\end{align*}
This and the assumptions on $\bfu_0$ imply
\begin{align*}
\E\bigg[\sup_{t\in(0,T)}\int_Gh_L(|\bfu|)\,dx\bigg]&\leq\,c\,\E\bigg[1+\sup_{t\in(0,T)}|(II)_q^3|+\sup_{t\in(0,T)}|(III)_q|\bigg]
\end{align*}
We have by (\ref{eq:phi}) and Lemma \ref{lems41}
\begin{align*}
\sup_{t\in(0,T)}|(III)_q|
&=\frac{1}{2}\sum_{k}\int_G\int_0^TD^2H_L(\bfu)\,d\Big\langle\int_0^{\cdot} g_k(\cdot,\bfu) d\beta_k\Big\rangle_\sigma\,dx\\
&\leq \frac{1}{2}\sum_{k}\int_G\int_0^T|D^2H_L(\bfu)||g_k(\cdot,\bfu) |^2\,d\sigma\,dx\\
&\leq c(q)\sum_{k}\int_G\int_0^T\tfrac{h_L'(|\bfu|)}{|\bfu|}|g_k(\cdot,\bfu)|^2\,d\sigma\,dx\\
&\leq c(q)\int_G\int_0^T h_L'(|\bfu|)|\bfu|^2\,d\sigma\,dx\\
&\leq c(q)\int_G\int_0^Th_L(|\bfu|)\big)\,d\sigma\,dx.
\end{align*}
 Similar to the proof of Theorem \ref{thm:2.3} we gain using $\big(\tfrac{h_L'(s)}{s}+h_L''(s)\big)s^2\leq c(\alpha)h_L(s)$ uniformly in $L$ (recall Lemma \ref{lems41} c))
\begin{align*}
&\,\E\,\bigg[\sup_{t\in(0,T)}|(II)_q^3|\bigg]\leq c\,\E\bigg[\bigg(\int_0^T\bigg(\int_G
h_L(|\bfu|)\,dx\bigg)^2\,dt\bigg]^{\frac 12}\\
&\leq \delta\,\E\bigg[\sup_{(0,T)}\int_G
h_L(|\bfu|)\,dx\bigg]+ c(\delta)\,\E\bigg[\int_\Q
h_L(|\bfu|)\,dx\,dt\bigg].
\end{align*}
Finally we have shown
\begin{align*}
\E\bigg[\sup_{t\in(0,T)}\int_Gh_L(|\bfu|)\,dx\bigg]&\leq\,c\,\E\bigg[1+\int_0^T\int_Gh_L(|\bfu|)\,dx\,dt\bigg]
\end{align*}
and by Gronwall's Lemma
\begin{align*}
\E\bigg[\sup_{t\in(0,T)}\int_Gh_L(|\bfu|)\,dx\bigg]&\leq\,c\,
\end{align*}
uniformly in $L$. Passing to the limit $L\rightarrow\infty$ yields the claim.
\end{proof}

\begin{proof} (of Theorem \ref{thm:4.3})
We apply It\^{o}'s formula to the function $$f_L(\bfv):=\int_G\eta^2H_L(\nabla\bfv)\,dx:=\int_G\eta^2h_L(|\nabla\bfv|)\,dx,$$ 
where $\eta\in C^\infty_0(G)$ is a cut-off function and $h_L$ is defined in (\ref{eq:hL}). We obtain
\begin{align*}
\int_G&\eta^2h_L(|\nabla\bfu|)\,dx\\&=\int_G\eta^2h_L(|\nabla\bfu_0|)\,dx+\int_0^t f'_L(\bfu)d\bfu_\sigma+\frac{1}{2}\int_0^t f''_L(\bfu)\,d\langle\bfu_\sigma\rangle_\sigma\\
&=\int_G\eta^2h_L(|\nabla\bfu_0|)\,dx+\int_G \int_0^t\eta^{2}DH_L(\nabla\bfu): d\nabla\bfu_\sigma\,dx\\&+\int_G\int_0^t \eta^{2}D^2H_L(\nabla\bfu) \,d\Big\langle\int_0^{\cdot}\nabla\big(\Phi(\bfu)\,d\bfW\big) \Big\rangle_\sigma\,dx\\
&=:(I)_\alpha+(II)_\alpha+(III)_\alpha.
\end{align*}
We consider the three integrals separately and decompose the second one into
\begin{align*}
(II)_\alpha&=-(II)^1_\alpha-(II)^2_\alpha-(II)^3_\alpha+(II)^4_\alpha,\\
(II)_\alpha^1&:=\int_0^t\int_G \eta^{2}\tfrac{h_L'(|\nabla\bfu|)}{|\nabla\bfu|}D\bfS(\nabla\bfu)\big(\partial_\gamma\nabla\bfu,\partial_\gamma\nabla\bfu\big),\\
(II)_\alpha^2&:=\int_0^t\int_G \tfrac{h_L'(|\nabla\bfu|)}{|\nabla\bfu|}D\bfS(\nabla\bfu)\big(\partial_\gamma\nabla\bfu,\nabla\eta^{2}\otimes\partial_\gamma\bfu\big)\,dx\,d\sigma,\\
(II)_\alpha^3&:=\int_0^t\int_G \eta^2D\bfS(\nabla\bfu)\big(\partial_\gamma\nabla\bfu,\nabla\tfrac{h_L'(|\nabla\bfu|)}{|\nabla\bfu|}\otimes\partial_\gamma\bfu\big)\,dx\,d\sigma,\\
(II)_\alpha^4&:=\int_0^t\int_G\eta^{2}\tfrac{h_L'(|\nabla\bfu|)}{|\nabla\bfu|}\nabla\bfu
:\nabla\big(\Phi(\bfu)\,d\bfW_\sigma\big)\,dx.
\end{align*}
Using the assumptions on $\bfS$, see (\ref{eq:Sp}), we obtain
\begin{align*}
(II)_\alpha^1
&\geq c\int_0^t\int_G \eta^{2}\tfrac{h_L'(|\nabla\bfu|)}{|\nabla\bfu|}(1+|\nabla\bfu|)^{p-2}
|\nabla^2\bfu|^2\,dx\,d\sigma.
\end{align*}
For the second term we gain for every $\delta>0$ using Young's inequality and Lemma \ref{lems41}
\begin{align*}
(II)_\alpha^2
&\leq \delta (II)_\alpha^1+c(\delta)\int_0^t\int_{\support{\eta}}\tfrac{h_L'(|\nabla\bfu|)}{|\nabla\bfu|}D\bfS(\nabla\bfu)\big(\nabla\eta^{2}\otimes\partial_\gamma\bfu,\nabla\eta^{2}\otimes\partial_\gamma\bfu\big)\,dx\,d\sigma\\
&\leq \delta (II)_\alpha^1+c(\delta)\int_0^t\int_{\support{\eta}} (1+|\nabla\bfu|)^{p-2}h_L(|\nabla\bfu|)\,dx\,d\sigma.
\end{align*}
Thanks to assumption (\ref{eq:nu}) and Lemma \ref{lems41} it holds\footnote{for a detailed explanation of this step we refer to \cite{Bi}, (32) on p. 62.}
\begin{align*}
(II)_\alpha^3&=\int_0^t\int_G \eta^2D\bfS\big(\partial_\gamma\nabla\bfu,\nabla\tfrac{h_L'(|\nabla\bfu|)}{|\nabla\bfu|}\otimes\partial_\gamma\bfu\big)\,dx\,d\sigma\\
&=\frac{1}{2}\int_0^t\int_G \eta^2D\bfS\big(e_\gamma\otimes\nabla\tfrac{h_L'(|\nabla\bfu|)}{|\nabla\bfu|},e_\gamma\otimes\nabla|\nabla\bfu|^2\big)\,dx\,d\sigma\\
&=\frac{1}{2}\int_0^t\int_G \eta^2\tfrac{h_L''(|\nabla\bfu|)|\nabla\bfu|-h_L'(|\nabla\bfu|)}{|\nabla\bfu|^3}D\bfS\big(e_\gamma\otimes\nabla|\nabla\bfu|^2,e_\gamma\otimes\nabla|\nabla\bfu|^2\big)\,dx\,d\sigma
\\&\geq0.
\end{align*}
Moreover, we have by (\ref{eq:phi}) and Lemma \ref{lems41}
\begin{align*}
(III)_\alpha
&=\frac{1}{2}\sum_{k}\int_G\int_0^t\eta^{2}D^2H_L(\nabla\bfu)\,d\Big\langle\int_0^{\cdot} \nabla\big(g_k(\cdot,\bfu) \big)d\beta_k\Big\rangle_\sigma\,dx\\
&\leq\frac{1}{2}\sum_{k}\int_G\int_0^t\eta^{2}|D^2H_L(\nabla\bfu)||\nabla\big(\cdot,g_k(\bfu) \big)|^2\,d\sigma\,dx\\
&\leq \sum_{k}\int_G\int_0^t\eta^{2}\Big(h_L''(|\nabla\bfu|)+\tfrac{h_L'(|\nabla\bfu|)}{|\nabla\bfu|}\Big) |\nabla\big( g_k(\cdot,\bfu)\big)\Big|^2\,d\sigma\,dx\\
&\leq c(\alpha+1)\sum_{k}\int_G\int_0^t\eta^{2}\tfrac{h_L'(|\nabla\bfu|)}{|\nabla\bfu|}\big(|\nabla_{\bfxi}g_k(\cdot,\bfu)\nabla\bfu|^2+|\nabla_x g_k(\cdot,\bfu)|^2\big)\,d\sigma\,dx\\
&\leq c(\alpha+1)\int_G\int_0^t\eta^{2}\tfrac{h_L'(|\nabla\bfu|)}{|\nabla\bfu|}\big(|\nabla\bfu|^2+|\bfu|^2\big)\,d\sigma\,dx\\
&\leq c(\alpha)\int_G\int_0^t\eta^{2}\big(1+h_L(|\nabla\bfu|)+h_L(|\bfu|)|\bfu|^2\big)\,d\sigma\,dx.
\end{align*}
In the last step we applied Lemma \ref{lems41} c) and
 d).Thus we obtain taking the supremum, the $q$-th power and applying expectations
\begin{align}\label{eq:hLq}
\begin{aligned}
\E&\bigg[\sup_{(0,T)}\int_G\eta^{2}h_L(|\nabla\bfu|)\,dx+\int_0^T\int_G \eta^{2}\tfrac{h_L'(|\nabla\bfu|)}{|\nabla\bfu|}(1+|\nabla\bfu|)^{p-2}
|\nabla^2\bfu|^2\,dx\,d\sigma\bigg]^q\\
&\leq c(\eta,\alpha)\,\E\bigg[1+\int_Gh_L(|\nabla\bfu_0|)\,dx+\int_0^T\int_{\support{\eta}}(1+|\nabla\bfu|)^{p-2}h_L(|\nabla\bfu|)\,dx\,d\sigma\bigg]^q\\
&+c(\eta,\alpha)\int_0^T\E\bigg[\int_G\eta^{2}\big(h_L(|\nabla\bfu|)+h_L(|\bfu|)|\bfu|^2\big)\,dx\bigg]^q\,d\sigma+c\,\E\bigg[\sup_{(0,T)}|(II)_\alpha^4(t)|\bigg]^q.
\end{aligned}
\end{align}
 Similar to the proof of Theorem \ref{thm:2.3} we gain using $\big(\tfrac{h_L'(s)}{s}+h_L''(s)\big)s^2\leq c(\alpha)h_L(s)$ uniformly in $L$ (recall Lemma \ref{lems41} c))
\begin{align*}
\,\E\,\bigg[\sup_{t\in(0,T)}|(II)_\alpha^4|\bigg]^q
&\leq \delta\,\E\bigg[\sup_{(0,T)}\int_G\eta^2
h_L(|\nabla\bfu|)\,dx\bigg]^q+ c(\delta)\,\int_0^T\E\bigg[\int_G\eta^2
h_L(|\nabla\bfu|)\,dx\bigg]^q\,dt\\
&+ c\,\E\bigg[\sup_{(0,T)}\int_G\eta^2
h_L(|\nabla\bfu|)\,dx\bigg]^q.
\end{align*}
If we choose $\delta$ small enough we can remove the term involving $(II)_\alpha^4$ from the right-hand-side of (\ref{eq:hLq}).
By Gronwall's Lemma, the assumptions on $\bfu_0$ and Lemma \ref{lems:alphau} we end up with
\begin{align}\label{eq:highq}
\E&\bigg[\int_G\eta^2h_L(|\nabla\bfu(t)|)\,dx+\int_\Q\eta^2\tfrac{h_L'(|\nabla\bfu|)}{|\nabla\bfu|}(1+|\nabla\bfu|)^{p-2}
|\nabla^2\bfu|^2\,dx\,dt\bigg]^q\\
&\leq c(\eta,\alpha)\E\bigg[1+\int_0^t\int_{\support{\eta}}(1+|\nabla\bfu|)^{p-2}h_L(|\nabla\bfu|)\,dx\,d\sigma\bigg]^q.
\nonumber
\end{align}
Assume for a moment that 
\begin{align}\label{eq:Lalpha}
\nabla\bfu\in L^q(\Omega,\F,\p;L^{p+\alpha}((0,T)\times G'))\quad\forall G'\Subset G,\,\, \forall q<\infty,\\
\bfu\in L^q(\Omega,\F,\p;L^{\infty}(0,T;L^{\alpha+2}(G')))\quad\forall G'\Subset G,\,\, \forall q<\infty,
\end{align}
Then we are allowed to go to the limit $L\rightarrow\infty$ on the r.h.s. of (\ref{eq:highq}). By Fatou's Theorem we are now allowed to do this on the l.h.s. as well. We obtain
\begin{align*}
\E&\bigg[\sup_{(0,T)}\int_G\eta^2h(|\nabla\bfu(t)|)\,dx+\int_\Q\eta^2|\nabla(1+|\nabla\bfu|)^{\frac{p+\alpha}{2}}|^2\,dx\,dt\bigg]^q\\&\leq c(\eta,\alpha)\bigg(1+\E\bigg[\int_0^T\int_{\support{\eta}}|\nabla\bfu|^{p+\alpha}\,dx\,dt\bigg]^q\bigg).
\end{align*}
This yields
\begin{align*}
|\nabla\bfu|^{\frac{p+\alpha}{2}}\in L^q(\Omega,\F,\p; L^2(0,T;W^{1,2}_{loc}(G)))\cap L^{q}(\Omega,\F,\p;L^\infty(0,T;L_{loc}^{2\frac{\alpha+2}{\alpha+p}}(G)))\quad \forall q<\infty.
\end{align*}
A parabolic interpolation (see for instance \cite{Am}, Thm. 3.1) shows on account of $p>2-\frac{4}{d}$
\begin{align}\label{eq:Lalpha'}
\nabla\bfu&\in L^q(\Omega,\F,\p;L^{\omega(\alpha)}(0,T)\times G'))\quad\forall G'\Subset G,\,\,\forall q<\infty,\\
\omega(\alpha)&:=(p+\alpha)\Big(1+\frac{2}{d}\frac{\alpha+2}{\alpha+p}\Big).
\end{align}
Since (\ref{eq:Lalpha}) is true for $\alpha=0$ (by Lemma \ref{lems:hm}) we start an iteration procedure by
\begin{align*}
\alpha_0:=0,\quad\alpha_{k+1}:=\omega(\alpha_k)-p,\quad k\in\N.
\end{align*}
On account of $\alpha_k\rightarrow\infty$ the claim is proven.
\end{proof}

\begin{remark}
As already observed in \cite{DiFr}, Remark 2.1., for the deterministic problem, it is not possible to obtain $L^\infty$-bounds for $\nabla \bfu$ except of the case $p=2$
via Moser iteration. So it is an open question if one can gain Lipschitz regularity
for the stochastic problem. In the deterministic case this is shown using the DeGiorgi method (see \cite{DiFr}, Lemma 2.3). However it is not clear if similar arguments will work for stochastic problems.
\end{remark}

\subsection*{Acknowledgement}
\begin{itemize}
\item The work of the author was supported by Leopoldina (German National Academy of Science).
\item The author wishes to thank the referee for the careful reading of the manuscript and many helpful advises.
\end{itemize}

\end{document}